\documentclass[11pt]{amsart}
\theoremstyle{plain}
\newtheorem{thm}{Theorem}[section]
\newtheorem{theorem}[thm]{Theorem}

\newtheorem{lemma}[thm]{Lemma}
\newtheorem{corollary}[thm]{Corollary}
\newtheorem{proposition}[thm]{Proposition}
\theoremstyle{definition}

\newtheorem{notation}[thm]{Notation}

\newtheorem{definition}[thm]{Definition}

\numberwithin{equation}{section}
\newcommand{\sF}{{\mathcal F}}

\newcommand{\sO}{{\mathcal O}}

\newcommand{\sZ}{{\mathcal Z}}
\newcommand{\C}{{\mathbb C}}

\newcommand{\BP}{{\mathbb P}}

\newcommand{\Z}{{\mathbb Z}}
\newcommand{\End}{{\rm End}}

\newcommand{\fg}{{\mathfrak g}}

\newcommand{\fgl}{{\mathfrak g}{\mathfrak l}}

\newcommand{\fo}{{\mathfrak o}}

\newcommand{\fv}{{\mathfrak v}}
\newcommand{\fl}{{\mathfrak l}}

\newcommand{\fa}{{\mathfrak a}}
\newcommand{\fb}{{\mathfrak b}}
\newcommand{\fh}{{\mathfrak h}}

\newcommand{\ft}{{\mathfrak t}}

\newcommand{\aut}{{\mathfrak a}{\mathfrak u}{\mathfrak t}}

\newcommand\p{\partial}
\newcommand\w{\widehat}
\newcommand\gr{\rm gr}

\def\Sym{\mathop{\rm Sym}\nolimits}

\def\Hom{\mathop{\rm Hom}\nolimits}

\title[Fundamental forms and infinitesimal symmetries]{Fundamental forms and infinitesimal symmetries of projective varieties}

\author{Jun-Muk Hwang and Qifeng Li}

\thanks{Jun-Muk Hwang was supported by the Institute for Basic Science (IBS-R032-D1). Qifeng Li was supported by the National Natural Science Foundation of China (Grant No. 12571043) and by the Natural Science Foundation of Shandong Province (Grant No. ZR2025QA08).}

\begin{document} 

\begin{abstract}
We give a bound on the dimension of the linear automorphism group of a projective variety $Z \subset \BP V$ in terms of its fundamental forms at a general point. Moreover, we show that the bound is achieved precisely when $Z \subset \BP V$ is projectively equivalent to an Euler-symmetric variety. As a by-product, we determine the Lie algebra of infinitesimal automorphisms of an Euler-symmetric variety and also obtain a rigidity result on the specialization of an Euler-symmetric variety preserving the isomorphism type of the fundamental forms.  
\end{abstract}

\maketitle

\medskip
MSC2010:  14M27, 14L30, 53A20

\section{Introduction}
We work over the field of complex numbers. Open subsets refer to those in the Euclidean topology. The projectivization $\BP V$ of a vector space $V$ is the set of one-dimensional subspaces of $V$.   

Recall the following from \cite[Definition 3.2]{FH20}. 

\begin{definition}\label{d.symbol}
Let $W$ be a vector space and let $\Sym W^* = \oplus_{k \in \Z} \Sym^k W^*$ be the graded ring of symmetric forms on $W$. Here, we use the convention $\Sym^k W^* =0$ if $k<0$.
\begin{itemize} \item[(i)] For $w \in W$ and $\zeta \in \Sym^k W^*$, let $\iota_w \zeta \in \Sym^{k-1} W^*$ be the contraction defined by $$\iota_w \zeta  (w_1, \ldots, w_{k-1}) := \zeta (w, w_1, \ldots, w_{k-1}).$$ \item[(ii)]  A graded vector subspace $$S = \oplus_{k \in \Z} S^k \subset \Sym W^*, \ S^k \subset \Sym^k W^*,$$ is called a {\em symbol system of rank $r$ on $W$} if  for each $k$, \begin{equation}\label{eq.iota} \iota_w s \in S^{k-1}  \mbox{ for any } w \in W \mbox{ and } s \in S^k, \end{equation} and $$S^0 = \C, \ S^1 = W^*, \ S^r \neq 0, \ S^{j} =0 \mbox{ for all } j>r.$$
 \end{itemize} \end{definition}

We have also the following natural equivalence relation.

\begin{definition}\label{d.isom} Let $S \subset \Sym W^*$ (resp. $\widetilde{S} \subset \Sym \widetilde{W}^*$) be a symbol system on a vector space $W$ (resp. $\widetilde{W}$). We say that the two symbol systems are {\em equivalent} if there is a linear isomorphism $f: W \to \widetilde{W}$ such that the induced isomorphism $f^*: \Sym \widetilde{W}^* \to \Sym W^*$ satisfies $f^*(\widetilde{S}) = S.$ \end{definition}

 The most important examples of symbol systems are systems of fundamental forms (see Definition \ref{d.FF}) at a general point of a projective variety, from the following result of E. Cartan (p.68 of \cite{LM03}).
 
 \begin{theorem}\label{t.Cartan}
 Let $Z \subset \BP V$ be a projective subvariety. Then for a general  point $x \in Z$, the system of fundamental forms $S_{x}$ at $x$, as a graded subspace of $\Sym T^*_x Z,$ is a symbol system on the tangent space $T_x Z$. If furthermore $Z$ is linearly nondegenerate in $\BP V$, then $\dim S_{x} = \dim V$ for a general point $x \in Z$. \end{theorem}
 
 Conversely,  \cite[Theorem 3.7]{FH20} says that a symbol system $S \subset \Sym W^*$ determines a projective variety $Z^S \subset \BP V^S$ (up to a projective equivalence), called the Euler-symmetric variety (see Definition \ref{d.Euler}), such that the system of fundamental  forms at any general point $x \in Z^S$ is equivalent to $S$. 
 
 Let $\w{Z} \subset V$ be the affine cone of a projective variety $Z \subset \BP V$. Let $\aut(\w{Z}) \subset \fgl(V) $ be the Lie algebra of the linear automorphism group $${\rm Aut}(\w{Z}):=\{ g \in {\rm GL}(V) \mid g(\w{Z}) = \w{Z} \}.$$ 
 Our main result is  the following relation between $\aut(\w{Z})$ and the system of fundamental forms of $Z$ at a general point.   
 
 \begin{theorem}\label{t.main}
 Let $Z \subset \BP V$ be a linearly nondegenerate projective variety and let $x \in Z$ be a general point such that (by Theorem \ref{t.Cartan}) the system of fundamental forms at $x$ is a symbol system $S_x \subset \Sym T^*_x Z$  of dimension equal to $\dim V$. Then $$ \dim \aut(\w{Z}) \ \leq \dim \aut(\w{Z}^{S_x})$$ and the equality holds if and only if $Z\subset \BP V$ is projectively equivalent to the Euler-symmetric variety $Z^{S_x} \subset \BP V^{S_x}$ by a linear isomorphism $V \cong V^{S_x}$. \end{theorem}
 
Euler-symmetric varieties are equivariant compactifications of vector groups (see \cite[Theorem 3.7]{FH20}), hence have large automorphism groups. According to Theorem \ref{t.main}, they have maximal linear automorphism groups among projective varieties admiting the same symbol algebras.
 
  The proof of Theorem \ref{t.main} consists of two steps. The first step involves  a suitable Lie algebra  of symmetries of the system of fundamental forms. 
 What should be the Lie algebra of symmetries of a symbol system $S \subset \Sym W^*$?
 From Definition \ref{d.isom}, a naive guess would be the subalgebra  consisting of elements of $\fgl(W)$ such that the induced endomorphisms of $\Sym^* W$  preserve both the gradation and the subspace $S \subset \Sym W^*$. It turns out that this is {\em not} the correct Lie algebra of symmetries: one has to consider hidden symmetries which do not preserve the gradation.  
 We give a correct definition in Definition \ref{d.fhS},  associating in a canonical way a finite-dimensional graded Lie algebra $\fh^S = \oplus_{k \in \Z} \fh^S_k$ to each symbol system $S.$  Then the first step in the proof of Theorem \ref{t.main} is the following. 
 
 \begin{theorem}\label{t.bound}
 Let $Z \subset \BP V$  and  $x \in Z$ be as in Theorem \ref{t.main} such that  the system of fundamental forms at $x$ is a symbol system $S_x \subset \Sym T^*_x M$  of dimension equal to $\dim V$. Then the Lie algebra $\fh := \aut(\w{Z})$ admits a filtration $$\fh = \fh_x^{-1} \supset \fh_x^0 \supset \fh^1_x \supset \cdots \supset \fh_x^k \supset \cdots,$$  such that $\cap_{k \in \Z} \fh^k_x =0$ and the graded object $\fh^k_x/\fh^{k+1}_x$ admits a natural inclusion into $\fh^{S_x}_k$ for each integer $k$. In particular, the dimension of $\aut(\w{Z})$ is at most $\dim \fh^{S_x}$. \end{theorem}

The second step in the proof of Theorem \ref{t.main} is the following.
 
 \begin{theorem}\label{t.Euler}
 In Theorem \ref{t.bound},  we have $\dim \aut(\w{Z}) = \dim \fh^{S_x}$ if and only if $Z \subset \BP V$ is projectively equivalent to the Euler-symmetric variety $Z^{S_x} \subset \BP V^{S_x}$.  In this case, the linear isomorphisms $\fh^k_x/\fh^{k+1}_x \cong \fh^{S_x}_k, k \in \Z,$ come from a Lie algebra isomorphism $\fh \cong \fh^{S_x}$. \end{theorem} 
 
 Clearly, Theorems \ref{t.bound} and \ref{t.Euler} imply Theorem \ref{t.main}. 
 We have the following application of Theorem \ref{t.main}.

 \begin{corollary}\label{c.limit}
 Fix a symbol system $S \subset \Sym W^*$. 
 Let $\Delta = \{ t \in \C \mid |t| < 1\}$ be the unit disc. Let $\sZ \subset \Delta \times \BP V$ be an irreducible closed analytic subset such that the fiber $$Z_t := \sZ \cap (\{t\} \times \BP V)   \ \subset \BP V $$ is projectively equivalent to the Euler-symmetric variety $Z^S \subset \BP V^S$ for all $t \in \Delta, t \neq 0$. Assume that for a general point $z$ in an irreducible component $Z'_0$ of $Z_0$, the system of fundamental forms $S_z \in \Sym T_z^* Z'_0$ is equivalent to $S \subset \Sym W^*$ as a symbol system. Then $Z'_0 \subset \BP V$ is projectively equivalent to $Z^S \subset \BP V^S$ and $Z_0 = Z'_0$. \end{corollary}  

\begin{proof}
The assumption on the system of fundamental forms at $z \in Z'_0$ implies that $\dim S_z = \dim V$, hence $Z'_0 \subset \BP V$ is linearly nondegenerate. Thus $\aut(\w{Z}_0) \subset \aut(\w{Z}'_0)$. Then by the upper-semi-continuity, we have $$\dim \aut(\w{Z}'_0) \ \geq \ \dim \aut(\w{Z}_0) \ \geq \ \dim \aut(\w{Z}_t)$$ for  $t \neq 0$. Thus Theorem \ref{t.main} implies that $Z'_0 \subset \BP V$ is projectively equivalent to $Z^S \subset \BP V^S$. It follows that $Z'_0 = Z_0$. \end{proof}

  Some special cases of Corollary \ref{c.limit} have been proved previously: \cite[Proposition 2.3]{Mo08},  \cite[Theorem 6.7]{HL21}, \cite[Corollary 4.7]{HL24} and \cite[Proposition 5.10]{HK25}.  
 These special cases play an important role in the problem of recognizing rational homogeneous spaces and horospherical varieties of Picard number 1 from their varieties of minimal rational tangents, because the latter are Euler-symmetric varieties. The proofs in these special cases used specific geometric properties of $Z^S$ in each case.  Our proof of Corollary \ref{c.limit} gives a uniform argument to cover all these cases, and potentially further examples of varieties of minimal rational tangents. This was our initial motivation for Theorem \ref{t.Euler}. 
 But Corollary \ref{c.limit} could be useful in a broader context because Euler-symmetric varieties form a large class of projective varieties. Among others, any subspace  $S^2 \subset \Sym^2 W^*$ gives a symbol system $S:= S^0 + S^1 + S^2$, hence an Euler-symmetric variety  $Z^S$. 
 For other interesting examples and a survey of results on Euler-symmetric varieties, see \cite[Section 5.2]{AZ22}.

  When $S_x \subset \Sym T^*_x M$ is isomorphic to the system of fundamental forms of an equivariantly embedded Hermitian symmetric space, Theorem \ref{t.bound} can be  deduced from the result of Se-ashi in \cite[Theorem 5.1.2]{Se88}.  Se-ashi's argument requires  special properties of the Lie algebra of the automorphism group of Hermitian symmetric spaces and  seems very difficult to be extended to the general setting of Theorem \ref{t.bound}. It was, however,  while we were examining his argument that we realized the expression of the Lie algebra $\fh^S$ and the filtration of $\fh$ in Theorem \ref{t.bound}. 
  We give a precise definition of the Lie algebra $\fh^S$  in Section \ref{s.fhS} and prove Theorem \ref{t.bound} and Theorem \ref{t.Euler} in Section \ref{s.FF} and Section \ref{s.Euler}, respectively.

\section{Infinitesimal symmetries of symbol systems}\label{s.fhS}

\begin{definition}\label{d.sym}
Let $W$ be a vector space. \begin{itemize} \item[(i)] Let $$ \fv(W) := \oplus_{k \in \Z} \fb_k (W), \ \fv_k(W):= (\Sym^{k+1} W^*) \otimes W $$ be the graded Lie algebra of polynomial vector fields on $W$. 
\item[(ii)] We have a natural representation of $\fv(W)$ on $\Sym W^*$:  an element $\sigma \otimes w \in \fv_i(W)$ with $ \sigma \in \Sym^{i+1} W^*$ and $w \in W$ sends   $\zeta \in \Sym^{k} W^*$ to $$ (\sigma \otimes w) \cdot \zeta := k \ \sigma \odot \iota_w \zeta,$$ where $\iota_w \zeta$ is the contraction from Definition \ref{d.symbol} and $\odot$ represents the symmetric product. This is equivalent to the representation of the Lie algebra of polynomial vector fields on the space of polynomial functions by derivative. 
    \item[(iii)] We write $\fa(W) = \Sym W^*$ when we regard $\Sym W^*$ as an abelian Lie algebra. Define $$\fg(W):= \oplus_{k \in \Z} \fg_k(W), \ \fg_k(W) := \fa_k(W) \oplus \fv_k(W)$$ as the semidirect product of the abelian Lie algebra $\fa(W)$ with the Lie algebra $\fv(W)$ via the representation in ${\rm (ii)}$. 
    \item[(iv)] 
 We have a natural representation of $\fg(W)$ on the space  $\Sym W^*$, where the subalgebra $\fa(W) = \Sym W^*$ acts on $\Sym W^*$ by the symmetric multiplication   and the subalgebra $\fv(W) \subset \fg(W)$ acts by ${\rm (ii)}$. This is equivalent to viewing $\fg(W)$ as the Lie algebra of differential operators of order less than or equal to $1$ on $W$ with polynomial coefficients and considering its action on the space of polynomial functions on $W$ as differential operators.   \end{itemize} \end{definition}
    
\begin{lemma}\label{l.injective}
The representation in Definition \ref{d.sym} ${\rm (iv)}$ gives an inclusion $$\fg_k \ \subset \ \Hom(\Sym^0 W^* + \Sym^1 W^*, \Sym^k W^* + \Sym^{k+1} W^*)$$ for each $k \in \Z$. \end{lemma}

\begin{proof}
In terms of a basis $x_1, \ldots, x_n, n= \dim W,$ of $W^*$, we can write an element $D$ of $\fg_k(W)$ as  $$D = h(x_1, \ldots, x_n) + \sum_{i=1}^n g^i(x_1, \ldots, x_n)  \frac{\p}{\p x_i}$$ where $h(x)$ is a homogeneous polynomial of degree $k$ and $g^i(x)$ is a homogeneous polynomial of degree $k+1$ for each $1 \leq i \leq n$. Suppose $D\cdot (\Sym^0 W^* + \Sym^1 W^*) =0$. Then 
$$0 = D \cdot x_j = h(x) x_j + g^j(x)$$ implies that $g^j(x) = -h(x) x_j$ for each $1\leq j \leq n$ and $D = h(x) (1 - \sum^n_{i=1} x_i \frac{\p}{\p x_i})$.  But then $0 = D \cdot 1 =  h(x)$ implies $D =0$. \end{proof}

\begin{notation}\label{n.weight}
Let $X$ be a complex manifold  and let $x \in X$ be a point.
\begin{itemize}
\item[(i)] For the ring $\sO_{X,x}$ of germs of holomorphic functions at $x$ and an integer $k$, we denote by $\mathbf{m}^k_{X,x}$  the ideal of functions vanishing at $x$ to order at least $k$. Here, we use the convention $\mathbf{m}^{k}_{X,x} = \sO_{X,x}$ for  $k \leq 0$. 
       \item[(ii)] For the Lie algebra $\ft_{X,x}$ of germs of holomorphic vector fields at $x$ and an integer $k$, let $\ft_{X,x}^k$ be the Lie subalgebra of vector fields vanishing at $x$ to order at least $k+1$.  Here, we use the convention $\ft^{k}_{X,x} = \ft_{X,x}$ for  $k \leq -1$. 
        \end{itemize} \end{notation}
        
        We skip the proof of the following elementary fact.
        
        \begin{lemma}\label{l.tm} For a complex manifold $X$ and a point $x \in X$, the following holds for all integers $k, \ell$. 
        \begin{itemize} \item[(i)] We have canonical identifications 
        $\mathbf{m}^{k}_{X,x}/\mathbf{m}^{k+1}_{X,x} = \Sym^{k} T^*_x X $  and $$\ft_{X,x}^{k}/\ft_{X,x}^{k+1} = (\mathbf{m}_{X,x}^{k+1}/\mathbf{m}_{X,x}^{k+2})\otimes T_x X =  (\Sym^{k+1} T^*_x X) \otimes T_x X.$$  \item[(ii)]  The Lie bracket of vector fields satisfies $[\ft^k_{X,x}, \ft^{\ell}_{X,x}] \subset \ft_{X,x}^{k + \ell}.$ In other words, the Lie algebra $\ft_{X,x}$ is a filtered Lie algebra under $$\ft_{X,x}= \ft_{X,x}^{-1} \supset \ft_{X,x}^0 \supset \ft_{X,x}^1 \supset \cdots.$$ 
        \item[(iii)] The derivatives of functions with respect to vector fields satisfy $\ft^k_{X,x} \cdot \mathbf{m}^{\ell}_{X,x} \subset \mathbf{m}^{\ell+ k}_{X,x}$. \end{itemize} \end{lemma}

\begin{definition}\label{d.linebundle}
Let $\pi: L \to M$ be a line bundle on a complex manifold $M$ and let $0_M \subset L$ be its zero section.
\begin{itemize} \item[(i)] For an open subset $U \subset M$, let $\sF(U)$ be the space of {\em $\pi$-linear functions} on $L|_U = \pi^{-1}(U)$, namely, holomorphic functions on $\pi^{-1}(U)$ which induce linear functionals on the 1-dimensional vector space $L_x =\pi^{-1}(x)$ for every $x \in U$. The germ $\sF_x$ can be identified with $\sO_{M,x} \otimes L^*_x$.  For each integer $k$, define the subspace
        $\sF_x^k \subset \sF_x$ by $$\sF_x^k := \mathbf{m}^k_{M,x} \otimes L^*_x \ \subset \ \sO_{M,x} \otimes L^*_x = \sF_x.$$
\item[(ii)] For an open subset $U \subset M$, let $\fl(U)$ be the Lie algebra of {\em $\pi$-linear vector fields} on $L|_U= \pi^{-1}(U)$, namely, holomorphic vector fields on $\pi^{-1}(U)$ which generate biholomorphic transformations of $\pi^{-1}(U)$ preserving the line bundle structure $\pi|_{\pi^{-1}(U)}: \pi^{-1}(U)  \to U$.     
     \item[(iii)] Regard a point $x \in M$ as a point of $0_M \subset L$. 
         We have a natural Lie algebra homomorphism $\rho: \fl_x \to \ft_{L,x}$  from the germ of $\fl$ at $x \in M$ to the germ of vector fields at $x \in L$. For each integer $k$, define
         $$\fl_x^k := \fl_x \cap \rho^{-1} (\mathbf{t}^k_{L,x}).$$ \end{itemize} 
\end{definition}

\begin{lemma} \label{l.weight} In Definition \ref{d.linebundle}, the following holds for each $x \in M$ and  all integers $k, \ell$.
 \begin{itemize}
  \item[(i)] There is a canonical identification $$\sF_x^k/\sF_x^{k+1} = (\mathbf{m}_{M,x}^{k}/\mathbf{m}_{M,x}^{k+1}) \otimes L^*_x = \Sym^{k} T^*_x M \otimes L^*_x.$$
\item[(ii)] Lie brackets of vector fields satisfy $[\fl_x^{k}, \fl_x^{\ell}] \subset \fl_x^{k + \ell}$ for all $k, \ell$. \item[(iii)] Taking derivative of functions with respect to vector fields determines a representation of $\fl_x$ on $\sF_x$ that satisfies $$\fl_x^k \cdot \sF_x^{\ell} \subset \sF_x^{k + \ell}.$$ 
    \item[(iv)] Let $\fo_x \subset \fl_x$ be the subalgebra consisting of  vector fields  in $\fl_x$ tangent to fibers of $\pi$. Then $\fo_x$ is an abelian ideal of the Lie algebra $\fl_x$ and there is a canonical isomorphism of vector spaces $\sO_{M,x} \cong \fo_x$. 
           \item[(v)] The differential  ${\rm d} \pi: T L \to T M$ of the map $\pi$  induces a Lie algebra homomorphism $\fl_x \to \ft_{M,x}$,    inducing a canonical exact sequence 
               $$ 0 \to \fo_x \subset \fl_x \stackrel{{\rm d} \pi}{\longrightarrow} \ft_{M,x} \to 0,$$ which yields the exact sequence $$0 \to \fo^k_x \to \fl^k_x \to \ft^k_{M,x} \to 0,$$ for each integer $k$, where $\fo_x^k \subset \fo_x$ is the subspace corresponding to $\mathbf{m}_{X,x}^k$ under the isomorphism $\fo_x\cong \sO_{M,x}$ in (iv).
               \item[(vi)] Taking graded objects of the filtration in ${\rm (v)}$ gives the exact sequence $$ \begin{array}{ccccccccc}  0 & \to & \fo_x^k/\fo_x^{k+1} & \to & \fl_x^k/\fl_x^{k+1} & \to & \ft_{M,x}^k/\ft_{M,x}^{k+1} & \to & 0 \\ & & \| & & \| & & \| & & \\ 0 & \to & \Sym^k T^*_x M & \to & \fl_x^k /\fl_x^{k+1} & \to & (\Sym^{k+1} T^*_x M) \otimes T_x M & \to & 0.\end{array} $$  
                              \item[(vii)] Fix a vector space $W$ of dimension equal to $\dim M$. Using the notation of Definition  \ref{d.sym}, we can choose isomorphisms $T_x X \cong W$ and  $ L^*_x \cong \C$ such that by the induced isomorphisms $$ \sF^k_x/\sF^{k+1}_x \cong \Sym^k W^*, \ \ft^k_{M,x}/\ft^{k+1}_{M,x} \cong \fv_k(W), \ \fl^k_x/\fl^{k+1}_x \cong \fg_k(W), $$   the representation of the graded Lie algebra  $\gr(\fl_x) = \oplus_{k\in \Z}  \fl^k_x/\fl^{k+1}_x $ on $\gr(\sF_x) = \oplus_{k \in \Z} \sF^k_x/\sF^{k+1}_x$ induced by the representation in ${\rm (iii)}$  can be identified with the representation of $\fg (W)$ on $\Sym W^*$ in Definition \ref{d.sym} ${\rm (vi)}$.                       
              % \item[(vi)] Since $\fo_x$ is an abelian ideal, the adjoint representation of $\fl_x$ on $\fo_x$ descends to a representation of the Lie algebra $\ft_{M,x}$ on $\fo_x$. Then the Lie bracket operation $[\ft_{M,x}, \fo_x] \to \fo_x$ coincides with the derivatives by vector fields $\ft_{M,x} \cdot \sO_{M,x} \to \sO_{M,x}$.   
                    \end{itemize} \end{lemma}

\begin{proof} 
Fix  a trivialization  $L|_U \cong U \times \C$  of the line bundle on an open neighborhood $U \subset M$ of $x$ equipped with a holomorphic coordinate system $z_1, \ldots, z_n, n= \dim M,$ and a linear coordinate $y$ on $\C$. Then an element of $\sF(U)$ is of the form $ f(z_1, \ldots, z_n) y$ for a holomorphic function $f$ on $U$ and an element of $\fl(U)$ is of the form $$\sum_{i=1}^n g^i(z_1, \ldots, z_n) \frac{\p}{\p z_i} + h(z_1, \ldots, z_n) y \frac{\p}{\p y}$$ for some holomorphic functions $g^i$ and $h$ on $U$.   It follows that the derivative of the function $f(z)y$ with respect to  this vector field is equal to the result of applying the differential operator $$\sum_{i=1}^n g^i(z_1, \ldots, z_n) \frac{\p}{\p z_i} + h(z_1, \ldots, z_n)$$ to the function $f(z)$. 
From Lemma \ref{l.tm} and the above coordinates description, all the statements (i) -- (vii) are easy to see. For example, the isomorphism $\sO_{M,x} \cong \fo_x$ in (iv) identifies $h(z) \in \sO_{M,x}$ with $h(z) y\frac{\p}{\p y} \in \fo_x$.  \end{proof}

 \begin{definition}\label{d.fhS}
 Let us use the notation of Definitions \ref{d.symbol} and \ref{d.sym}. For a symbol system $S = \oplus_{k \in \Z} S^k \subset \Sym W^*$ of rank $r \geq 1$, using the representation of $\fg(W)$ on $\Sym W^*$ in Definition \ref{d.sym}, define a Lie subalgebra $\fh^S \subset \fg(W)$ by 
 $$\fh^S := \{ A \in \fg(W) \mid A \cdot S \subset S\}.$$   We call $\fh^S$ the {\em Lie algebra of infinitesimal symmetries} of the symbol system $S$.  \end{definition}

 \begin{lemma}\label{l.fh}
 Set $\fh^S_k := \fh^S \cap \fg_k(W)$ in Definition \ref{d.fhS}. Then \begin{itemize}  \item[(i)]
  $\fh^S_{-1}= \fg_{-1}(W) =  W$ ; \item[(ii)]  $\fh^S_k = 0 \mbox{ if } k \leq -2 \mbox{ or } k > r$; and \item[(iii)] $\fh^S= \oplus_{k \in \Z} \fh^S_k $ is a graded Lie subalgebra of $\fg$. \end{itemize} \end{lemma}
 
 \begin{proof} The equality $\fh^S_{-1}= \fg_{-1}(W) = \fv_{-1}(W) = W$ is immediate from (\ref{eq.iota}), proving (i). In (ii), it is obvious that $\fh^S_k = \fg_k(W) =0$ for $k \leq -2$.  By Lemma \ref{l.injective},   a nonzero element of $\fg_{r+i}(W), i >0$ cannot send $$S^0 + S^1 = \Sym^0 W^* + \Sym^1 W^* \mbox{ to } S^{r+i} + S^{r+i+1} =0.$$ It follows that  $\fh^S_{r+i} =0$ for $i >0$, proving (ii). Then (i) and (ii) imply   $$\fh^S=\fg_{-1}(W)\oplus (\fh^S\cap\oplus_{k\geq 0}\fg_k(W))$$ and it is easy to see that $\fh^S\cap\oplus_{k\geq 0}\fg_k(W)=\oplus_{k\geq 0}\fh^S_k$, proving (iii).  \end{proof}  
 
 \section{Fundamental forms and infinitesimal symmetries}\label{s.FF}
    
    \begin{definition}\label{d.aut}
    Let $Z \subset \BP V$ be a projective variety in the projectivization of a vector space $V$.  Let $M \subset Z$ be the smooth locus of $Z$ and let $\w{M} \subset \w{Z} \subset V$ be their affine  cones. 
Let $\pi: L \to M$ be the tautological line bundle on $M$ with the fiber $L_x$ equal to the 1-dimensional subspace $\w{x} \subset V$ corresponding to $x \in M \subset \BP V$. 
Write $T_v \w{M} \subset V$ for the affine tangent space of $\w{M}$ at a point $0 \neq v \in \w{M}$. \begin{itemize} \item[(i)] Define $$\aut(\w{Z}) = \aut(\w{M}) :=\{ A \in \fgl(V) \mid A \cdot v \in T_v \w{M} \mbox{ for all } 0 \neq v \in \w{M}\},$$ called the {\em Lie algebra of infinitesimal linear automorphisms} of $Z$. It is the Lie algebra of the Lie group $$\{ g \in {\rm GL}(V) \mid g\cdot \w{M} = \w{M} \}.$$ \item[(ii)] We have a natural morphism $\beta: L \to \w{M} \subset V$ contracting the zero section $0_M$ to $0 \in V$ and sending $L_x$ to $\w{x}$ isomorphically. For $A \in \aut(\w{M})$, denote by $\vec{A}$  the unique vector field on $L$ such  that  for $w \in L \setminus 0_M$ and $v = \beta(w), $ $${\rm d}_w \beta (\vec{A}_w) = A \cdot v \ \in T_v \w{M}.$$ \end{itemize} \end{definition}

    \begin{proposition}\label{p.psi}
    In Definition \ref{d.aut}, for each $A \in \aut(\w{M})$, the vector field $\vec{A}$ on $L$ belongs to $\fl(M)$,  namely, it is a $\pi$-linear vector field on $L$ in the sense of Definition \ref{d.linebundle} ${\rm (ii)}$. This determines a Lie algebra homomorphism $$ \Psi_x: \aut(\w{M}) \ \longrightarrow \ \fl_x$$ for each $x \in M$. If furthermore $M \subset \BP V$ is linearly nondegenerate, then $\Psi_x$ is injective. \end{proposition}
    
    \begin{proof}
    The 1-parameter subgroup $\exp(\C A) \subset {\rm GL}(V)$ preserves $\w{M}$ and the line bundle structure $L \to M$. Thus $\vec{A}$ belongs to $\fl(M)$, defining the homomorphism $\Psi_x$ for any $x \in M$. Assuming that $M$ is linearly nondegenerate, suppose $\Psi_x(A) =0$ for some $A \in \aut(\w{M}).$ Then $A \cdot \beta(w) =0$ for any $w \in L $ over a neighborhood of $x \in M$. This implies that $A =0$ by the linear nondegeneracy, proving the injectivity of $\Psi_x$. \end{proof}  
    
    \begin{definition}\label{d.fhx}
 In Proposition \ref{p.psi}, write $\fh= \aut(\w{M})$ and pick  a point $x \in M$.  \begin{itemize} \item[(i)]  Define for each integer $k$,
 $$\fh_x^k := \Psi_x^{-1}( \fl_x^k )$$ such that 
$\fh = \fh_x^{-1} \supset \fh_x^0 \supset \fh_x^1 \supset \cdots $ gives a structure of filtered Lie algebra on $\fh$. We  denote by $\fh_x$ this filtered algebra. \item[(ii)] Let $$\gr(\fh_x) := \oplus_{k \in \Z} \gr(\fh_x)_k, \  \gr(\fh_x)_{k}:= \fh_x^k/\fh_x^{k+1}$$ be the associated graded Lie algebra. 
    \item[(iii)]     
     Via Lemma \ref{l.weight} (iii), we have a  representation of $\fh_x$ on $\sF_x$  induced by $\Psi_x$, which is compatible with the filtrations.  Write  $$\Phi_x: \gr(\fh_x) \to {\rm End}(\gr(\sF_x))$$ for the associated  representation, which is compatible with the gradations. \end{itemize}
    \end{definition}

 \begin{lemma}\label{l.bound} Assume that $M$ is linearly nondegenerate in $\BP V$.   Then \begin{itemize} \item[(i)] the homomorphism $\Phi_x$ in Definition \ref{d.fhx} is injective; \item[(ii)] 
 $\fh_x^m =0$ for some positive integer $m$; and \item[(iii)]  $\dim \fh = \dim \gr(\fh_x).$ \end{itemize} \end{lemma}
 
 \begin{proof}
 The injectivity of $\Psi_x$ in Proposition \ref{p.psi} shows that the induced homomorphism $\gr(\fh_x) \to \gr(\fl_x)$ is injective. Since $\gr(\fl_x)$ acts on $\gr(\sF_x)$ faithfully, we obtain (i).  For each $x \in M,$ there is some integer $m$ such that  the vanishing order of a nonzero vector field $\vec{A} \in \aut(\w{M})$  is bounded by  $m$. Thus (ii) follows also from the injectivity of $\Psi_x$. Finally, (iii) is immediate from (ii). \end{proof}   
 
    \begin{definition}\label{d.FF}
Let us use the terminology of Definitions \ref{d.linebundle} and \ref{d.aut}.
 \begin{itemize} \item[(i)] Via the morphism $\beta: L \to V$, each element $\lambda \in V^*$ determines a $\pi$-linear function $\lambda_L \in \sF(M)$. For $x \in M$, we write $\lambda_x \in \sF_x$ for the germ of $\lambda_L$ at $x \in M$.  \item[(ii)]  For each integer $k$, define $$V_x^{*k} := \{ \lambda \in V^* \mid \lambda_x \in \sF_x^k \}$$ and let $r$ be the smallest integer satisfying $V_x^{*r+1} = V_x^{*r+i}$ for all $i \geq 1$,  inducing a filtration $$V^* = V_x^{*0} \supset V_x^{*1} \supset V_x^{*2} \supset \cdots \supset V_x^{*r} \supset V_x^{*r+1}$$ satisfying $\cap_{k \in \Z} V_x^{*k} = V_x^{* r+1}$.  \item[(iii)] Using Lemma \ref{l.weight}, we have the induced homomorphism $${\rm FF}^k_x: V_x^{*k}/V_x^{* k+1} \to \sF_x^{k}/\sF_x^{k+1} = \Sym^k T^*_x M \otimes L_x^*,$$  called the {\em $k$-th fundamental form} of $M$ at $x \in M$. 
\item[(iv)] For each $k$, let $S^k_{x} \subset \Sym^k T^*_x M$ be the image of the $k$-th fundamental form.  The graded subspace $$S_{x} = \oplus_{k \in \Z} S^k_{x} \subset \Sym T^*_x M$$ is called the {\em system of fundamental forms} of $M$ (or its closure $Z$) at $x \in M$. \end{itemize} \end{definition}

    \begin{lemma}\label{l.derivative}
    For $A \in \fgl(V)$, we have the dual representation, namely, for $\lambda \in V^*$, the evaluation of $A \cdot \lambda \in V^*$ at $v \in V$ is $$ \langle A \cdot \lambda, v \rangle = - \langle \lambda, A \cdot v \rangle.$$ Let $A \in \aut(\w{M})$ and let $\vec{A} \in \fl(M)$ be the $\pi$-linear vector field on $L$ from Definition \ref{d.aut}. For  $\lambda \in V^*$ and  $\lambda_L \in \sF(M)$ in Definition \ref{d.FF},  we have $$\vec{A} \cdot \lambda_L = - (A \cdot \lambda)_L,$$
    where the left hand side denotes the derivative of the $\pi$-linear  function $\lambda_L$ on $L$ with respect to the vector field $\vec{A}$ and the right hand side denotes the $\pi$-linear function on $L$ corresponding to the element $-A \cdot \lambda \in V^*$. \end{lemma}
    
    \begin{proof} For any point $w \in L \setminus 0_M$ and $v = \beta(w) \in \w{M} \setminus 0$, choose an arc $\gamma: \Delta := \{ t \in \C \mid |t| < \epsilon\}  \to \w{M}$ with the tangent vector $A\cdot v \in T_v \w{M}$ of the form $$\gamma(t) = v + t A \cdot v + O(t^2), \ t \in \Delta.$$ Then \begin{eqnarray*} (\vec{A} \cdot \lambda_L) (w) &=& {\rm d}_w \pi (\vec{A})_v (\lambda|_{\w{M}}) (v) \\ &=&  \frac{\rm d}{{\rm d} t}|_{t=0} \lambda (v + t A \cdot v + O(t^2)) \\ &=& \langle \lambda, A \cdot v \rangle \\ &=& - \langle A\cdot \lambda,  v \rangle \\ &=& - (A \cdot \lambda)_L (w).\end{eqnarray*}  This proves the lemma. \end{proof} 
     
\begin{lemma}\label{l.germ}
For $A \in \aut(\w{M}) = \fh$ and $x \in M$, if $A \in \fh^k_x$ for some $k \in \Z$ in the notation of Definition \ref{d.fhx}, then $A \cdot V_x^{*\ell} \subset V_x^{* \ell+k}$ for any $\ell \in \Z$ in the notation of Definition \ref{d.FF}.  \end{lemma}

\begin{proof} For $\lambda \in V_x^{* \ell}$, the germ $\lambda_x \in \sF_x$ of $\lambda_L$  at $x$ belongs to $\sF^{\ell}_x$. Then Lemma \ref{l.weight} (iii) implies $\vec{A}_x \cdot \lambda_x \in \sF^{\ell + k}_x$. Thus $A \cdot \lambda \in V_x^{*\ell +k}$ by Lemma \ref{l.derivative}. \end{proof}  
 
 We are ready to prove Theorem \ref{t.bound}.
 
 \begin{proof}[Proof of Theorem \ref{t.bound}] Let $M$ be the smooth locus of $Z \subset \BP V$ and let $x \in M$ be as in Theorem \ref{t.bound}. We use the filtrations in Definitions \ref{d.fhx} and \ref{d.FF}.  For $A \in \fh_x^k$ satisfying $ A \notin \fh_x^{k+1}$, let $ 0 \neq \gr(A) \in \gr(\fh_x)_k$ be the corresponding element. Then for each integer $\ell,$ Lemma \ref{l.derivative} and Lemma \ref{l.germ} give a commutative diagram
$$\begin{array}{ccc} V_x^{*\ell}/V_x^{*\ell+1} & \stackrel{-A}{\longrightarrow} & V_x^{* \ell +k}/V_x^{* \ell + k +1} \\ {\rm FF}^{\ell}_x \downarrow & & \downarrow {\rm FF}^{\ell+k}_x \\ \sF_x^{\ell}/\sF_x^{\ell+1} & \stackrel{\Phi_x(\gr(A))}{\longrightarrow} & \sF_x^{\ell+k}/\sF_x^{\ell+ k +1}, \end{array} $$ where $\Phi_x$ is the representation in Definition \ref{d.fhx}.  
 Thus $\Phi_x(\gr(A))$ sends $S_x^{\ell}$ to $S_x^{\ell +k}$ for all $\ell$.
 
 Fix a vector space $W$ of dimension equal to $\dim M$.
 Choosing an isomorphism $T_x M \cong W,$ we can identify $\gr(\fl_x)$ with $\fg(W)$ and $S_x = \oplus_{k \in \Z} S_x^k \subset \Sym T^*_x M$ with a symbol system $S = \oplus_{k \in \Z} S^k \subset \Sym W^*$.  Then the image of $\Phi_x(\gr(A))$ lies in $\fg_k(W)$ and since it sends $S^{\ell}$ to $S^{\ell +k}$ for all $\ell$, it  belongs to $\fh^{S}_k \subset \fg_k(W)$. Consequently,  we have a homomorphism $\fh_x^k/\fh_x^{k+1} \to \fh^{S}_k$ for each $k \in \Z$.
 
By the assumption that $Z \subset \BP V$  is linearly nondegenerate, the homomorphism  $\Phi_x$ is injective from Lemma \ref{l.bound} (i) and so is the above homomorphism $\fh_x^k/\fh_x^{k+1} \to \fh^{S}_k$ for each $k \in \Z$. Then (ii) and (iii) of Lemma \ref{l.bound} give $\cap_{k\in \Z} \fh^k_x = 0$ and  $\dim \fh = \dim \gr(\fh_x) \leq \dim \fh^S$, completing the proof of the theorem. \end{proof}

 \section{Infinitesimal symmetries of Euler-symmetric varieties}\label{s.Euler}
 
 We recall the following from  \cite[Definition 3.6]{FH20}. In \cite[Definition 2.1]{FH20}, Euler-symmetric varieties are defined in terms of $\C^*$-actions. But the two definitions are equivalent by \cite[Theorem 3.7]{FH20}.
 
 \begin{definition}\label{d.Euler}
 Let $S = \oplus_{k\in \Z} S^k \subset \Sym W^*$ be a symbol system of rank $r$. \begin{itemize}
 \item[(i)] Define $S_\bot:= \oplus_{k \in \Z} S^k_\bot \subset \Sym W$ with   
$$ S^k_\bot:=\{\xi\in\Sym^k W \mid f(\xi)=0 \mbox{ for all } f\in S_k\}.$$
\item[(ii)] Define $V^S:= \oplus_{k \in \Z} V^S_k$ with $V^S_k:=\Sym^k W/S^k_\bot$.  Note that $V^S_0= \C, V^S_1 = W$ and $V^S_k =0$ for $k<0$ or $k >r$. We can  identify $V^S_k$ with $(S^k)^*$ and $V^S$ with $S^*$. 
\item[(iii)] For $\xi \in \Sym^k W$, denote by $[\xi] \in V^S_k$  its coset class modulo $S^k_\bot.$ Define a morphism  $\phi^S: \C \oplus W \to \oplus_{k \in \Z} V^S_k = V^S$ by sending $(t, v) \in \C \oplus W$ to $$\phi^S(t, v)  := \sum_{k \in \Z} [t^{r-k} v^k]  \ \in \oplus_{k \in \Z} V^S_k = V^S.$$ Here, note that when $k> r,$ we have $S^k_\bot=\Sym^k W$ and thus $[t^{r-k} v^k]\in V^S_k=0$ for all $t\in \C$ and $v\in W$. In particular, the summand is nonzero only for  $0 \leq k \leq r$ and we need not worry about the meaning of $t^i, v^i$ for $i<0$. 
\item[(iv)]  The closure of the image of the morphism $\phi^S$ is denoted by $\w{Z}^S \subset V^S$. The projective subvariety  $Z^S \subset \BP V^S$ whose affine cone is $\w{Z}^S$ is called the {\em Euler-symmetric variety} determined by the symbol system $S$. \end{itemize} \end{definition}
 
Recall the following properties of Euler-symmetric varieties from \cite[Definition 2.1 and Theorem 3.7]{FH20}. 

\begin{proposition}\label{p.C*}
A projective variety $Z \subset \BP V$ is projectively equivalent to $Z^S \subset \BP V^S$ in Definition \ref{d.Euler} if and only if the following three conditions are satisfied. \begin{itemize}
\item[(i)] $\dim S = \dim V^S =  \dim V$. 
\item[(ii)] The system of fundamental forms $S_z \subset \Sym T^*_z Z$ at a general point $z \in Z$ is isomorphic to $S \subset \Sym W^*$.
    \item[(iii)] For a general point $z \in Z$, there is a multiplicative subgroup $\C^* \subset {\rm GL}(V)$ that preserves $\w{Z}$ and  $\w{z} \subset \w{Z}$ such that the induced $\C^*$-action on the tangent space $T_z Z$ is by scalar multiplication. \end{itemize} \end{proposition} 
 
 \begin{lemma}\label{l.tangent}
In Definition \ref{d.Euler}, 
for $u, v \in W$, define
\begin{equation}\label{eq.eta}
\eta(u, v):= \sum_{k\in \Z} k \cdot [u\odot v^{k-1}] \ \in \ V^S.
\end{equation}
Then for any choice of $v \in W$ and
$$
\xi:=\phi^S(1, v)= \sum_{k \in \Z} [v^k] \ \in \    \w{Z}^S,
$$
 the affine tangent space $T_\xi(\w{Z}^S) \subset V^S$ is spanned by $\xi\in V^S$ itself and the vectors $$\{ \eta(u, v)\in V \mid  u \in W\}.$$
\end{lemma}

\begin{proof}
This is immediate from 
$$
\lim\limits_{\epsilon\rightarrow 0}\frac{\phi^S(1, v+\epsilon u)-\phi^S(1, v)}{\epsilon}=\eta(u, v)\in V^S. $$
\end{proof}

\begin{lemma}\label{l.dual}
The representation of $\fg(W)$ on $\Sym W^*$ in Definition \ref{d.sym} has its dual representation on $\Sym W$. The following holds for the dual representation.
\begin{itemize} \item[(i)] If $A \in \fg_k(W)$, then $A \cdot \Sym^{\ell} W \subset \Sym^{\ell-k} W$.
\item[(ii)] If $A = y^k \in \fa_k(W)$ for some $y \in W^*$, then $A \cdot v^k = - y(v)^k$ for any $v \in W$.  
\item[(iii)] If $A \in \fa_k(W)$ and $j \geq k$, then $A \cdot v^j = (A \cdot v^k) v^{j-k}$ for any $v \in W$.
       \item[(iv)] If $A = y^{k+1} \otimes u \in \fv_k(W)$ for some $y \in W^*$ and $u \in W$, then $A \cdot v^{k+1} = - y(v)^{k+1} u$ for any $v \in W$. 
    \item[(v)] If $A \in \fv_k(W)$ and $j \geq k$, then $A \cdot v^j = (j-k) 
     (A \cdot v^{k+1}) \odot v^{j-k-1}$ for any $v \in W$. Here, the right hand side is zero  when $j=k$. 
 \end{itemize}\end{lemma} 
 
 \begin{proof}
 From Definition \ref{d.sym}, we know that the representation of $\fg(W)$ on $\Sym W^*$ satisfies $$\fg_k(W) \cdot \Sym^{\ell} W^* \subset \Sym^{\ell + k} W^*$$ for all $k, \ell \in \Z$. This implies (i) by duality. 
 
 (ii) follows from \begin{eqnarray*} \langle 1, A \cdot v^k \rangle & = & - \langle A \cdot 1, v^k\rangle \\ &=& - \langle y^k, v^k \rangle \ = \ - y(v)^k. \end{eqnarray*}
 
 To check (iii), we may assume by linearity $A = y^k$ for some $y \in W^*$. Setting $s= j-k$, we need to show $\langle \xi, A \cdot v^j\rangle = \langle \xi, (A \cdot v^k) v^s \rangle$ for any $\xi \in \Sym^s W^*$. By linearity, we may check it for $\xi = x^s$ with  $x \in W^*$. Then \begin{eqnarray*}
 \langle x^s, A \cdot v^j \rangle & = & - \langle A \cdot x^s, v^j\rangle \\
 & = & - \langle y^k \cdot x^s, v^j\rangle \\ &=& - y(v)^k x(v)^s \\ &=& \langle x^s, - y(v)^k v^s \rangle \\ &=& \langle x^s, (A \cdot v^k) v^s \rangle, \end{eqnarray*} where the last equality uses (ii). This proves (iii). 
 
 (iv) is a consequence of the following equality for any $z \in W^*$: \begin{eqnarray*}
 \langle z, A \cdot v^{k+1} \rangle  &=& - \langle A \cdot z, v^{k+1} \rangle \\
 &=& - \langle (y^{k+1} \otimes u) \cdot z, v^{k+1} \rangle \\
 &=& - \langle y^{k+1} z(u), v^{k+1} \rangle \\ & = & - y(v)^{k+1} z(u) \\ & = & \langle z, -y(v)^{k+1} u \rangle. \end{eqnarray*} 
 
 To check (v), as in the proof of (iii), we may check $$\langle \xi, A \cdot v^j \rangle = (j-k) \langle \xi, (A \cdot v^{k+1}) \odot v^{j-k-1} \rangle$$ for any $\xi = x^s$ with $s = j-k$ and $x \in W^*$. Moreover, we may check it for $A = y^{k+1} \otimes u$ with $y \in W^*$ and $u \in W$. But \begin{eqnarray*} \langle  x^s, A \cdot v^j \rangle  &=& - \langle A \cdot x^s, v^j \rangle \\ &=& -\langle (y^{k+1}\otimes u) \cdot x^s, v^j\rangle \\ &=& - \langle s x(u) y^{k+1}x^{s-1}, v^j \rangle \\ &=& - s x(u) y(v)^{k+1} x(v)^{s-1} \\ &=& \langle x^s, - s y(v)^{k+1} v^{s-1} \odot u \rangle \\ &=& \langle x^s, s v^{s-1} \odot (A \cdot v^{k+1}) \rangle, \end{eqnarray*}  where the last equality uses (iv). Furthermore, when $s = j-k=0$, the right hand side becomes zero. This proves (v). \end{proof}

\begin{lemma}\label{l.action}
Fix a symbol system $S \subset \Sym W^*$ and let $\fh^S \subset \fg(W)$ be the Lie algebra of infinitesimal symmetries of $S$. \begin{itemize} \item[(i)] Under the natural representation of $\fg(W)$ on $\Sym W^*,$ the subalgebra   $\fh^S \subset \fg(W)$ preserves the subspace  $S \subset \Sym W^*$, inducing a faithful representation of $\fh^S$ on $S$. \item[(ii)] The dual representation of $\fg(W)$ on $\Sym W$ gives a faithful representation of $\fh^S$ on $\Sym W$, which preserves $S_\bot \subset \Sym W$. It induces a representation of $\fh^S$ on $V^S$, which is dual to the representation on $S$ in ${\rm (i)}$.  \end{itemize} \end{lemma}

\begin{proof}
 We only verify the faithfulness of the representation of $\fh^S$ on $S$, and other statements are straightforward.
By Lemma \ref{l.injective}, the Lie algebra $\fg(W)$, hence its subalgebra $\fh^S$, is a vector subspace of $\Hom(\oplus_{k=0}^1\Sym^k W^*, \Sym W^*)$. Since $S^k=\Sym^k W^*$ for $k=0, 1$, the natural homomorphism 
$$\fh^S\rightarrow\End(S)\rightarrow\Hom(S^0\oplus S^1, S)$$
is injective. It follows that the representation of $\fh^S$ on $S$ is faithful. \end{proof} 

\begin{proposition}\label{p.key}
Let $S \subset \Sym W^*$ and $\w{Z}^S \subset  V^S$ be as in Definition \ref{d.Euler}. We can regard $\fh^S$ as a subalgebra of $\fgl(V^S)$ by Lemma \ref{l.action}. Then  for any $A \in \fh^S_k, k \in \Z $ and a general $\xi\in\widehat{Z}^S$,  we have $A \cdot \xi \in T_\xi \widehat{Z}^S$. It follows that $A \in \aut(\w{Z}^S) \subset \fgl(V^S).$   
\end{proposition}

\begin{proof} By the invariance of $\w{Z}^S \subset V^S$ under the scalar multiplication on $V$, it suffices to check $A \cdot \xi \in T_{\xi} \w{Z}^S$  for any $v \in W$ and $\xi:= \phi^S(1, v) \in \w{Z}^S.$  
Recall that $$ \xi = \sum_{j\in \Z} [v^j] \ \in V^S.$$
By Lemma \ref{l.dual}, we have $A \cdot V^S_j \subset V^S_{j-k}$. Thus \begin{equation}\label{eq.k-1} A\cdot [v^j] = [A \cdot v^j] =0 \mbox{ if } j \leq k-1. \end{equation}
Applying the direct sum decomposition $$\fg_k(W)=\Sym^k W^*\oplus (\Sym^{k+1}W^*\otimes W),$$ we can write
$$
A=A'+A'' \ \in \ \fh^S_k \subset \fg_k(W),
$$
with $A'\in \Sym^k W^*$ and $A''\in \Sym^{k+1}W^*\otimes W$. Set
\begin{eqnarray*}
c & := & A'\cdot  v^k \ \in \ \C, \\
u & := & A''\cdot  v^{k+1} \ \in \ W.
\end{eqnarray*}
Lemma \ref{l.dual} (v) shows $A'' \cdot v^k = 0$, hence \begin{equation}\label{eq.vk} A \cdot [v^j] = c \mbox{ if } j=k.\end{equation} 
For $j\geq k+1$, we have by Lemma \ref{l.dual}
\begin{eqnarray*}
 A\cdot[v^j] &=& [A'\cdot v^j]+[A''\cdot v^j] \\
&=& [(A' \cdot v^k) v^{j-k}]+(j-k) [(A''\cdot v^{k+1})\odot v^{j-k-1}] \\
&=& c [v^{j-k}]+(j-k) [ u\odot v^{j-k-1}] \ \in \ V^S_{j-k}.
\end{eqnarray*}
In particular, one can see that when $j>\max\{r, k\}$ the vector $c v^{j-k}+(j-k) u\odot v^{j-k-1}\in S^{j-k}_\bot$ and thus the corresponding coset class in $V_{j-k}=\Sym^{j-k} W/S^{j-k}_\bot$ is zero.
Combining the above equality   with (\ref{eq.k-1}) and (\ref{eq.vk}), we have 
\begin{eqnarray*}
A \cdot \xi &=& A\cdot(\sum_{j \in \Z} [v^j]) \\
&=& \sum_{j \in \Z}[A\cdot v^j] \\
&=& \sum_{j=k}^{\infty} c [v^{j-k}]+\sum_{j=k+1}^{\infty} (j-k) [u\odot v^{j-k-1}] \\
&=& c \sum_{i=0}^{\infty} [v^i] + \sum_{i=0}^{\infty} (i+1) [u \odot v^i] \\
&=& c \xi+\eta(u, v) \ \in V^S , 
\end{eqnarray*} where we use (\ref{eq.eta}) in the last line. 
 By Lemma \ref{l.tangent}, this implies $A\cdot \xi \in T_{\xi} \w{Z}^S$, proving the proposition.  \end{proof}
 
 \begin{proposition}\label{p.aut}
For the Euler-symmetric variety $Z^S \subset \BP V^S$ associated with a symbol system $S \subset \Sym W^*$, the Lie algebra homomorphism $\fh^S \to \aut(\w{Z}^S)$ obtained from Proposition \ref{p.key} is an isomorphism. 
\end{proposition}

\begin{proof}
We know that the homomorphism $\fh^S \to \aut(\w{Z}^S)$ is injective because $\fh^S \subset \fgl(V^S)$ from Lemma \ref{l.action}.  By Proposition \ref{p.C*},  the system of fundamental forms of $Z^S$ at a general point is isomorphic to $S$ as a symbol system. Consequently, Theorem \ref{t.bound} implies $\dim \aut(\w{Z}^S) \leq \dim \fh^S$. Hence the homomorphism must be surjective.  
\end{proof}

\begin{proof}[Proof of Theorem \ref{t.Euler}]
Proposition \ref{p.aut} shows that $\aut(\w{Z}^S) \cong \dim \fh^S$ for the Euler-symmetric variety $\w{Z}^S \subset \BP V^S$. Conversely, suppose that $Z \subset \BP V$ is linearly nondegenerate and $\dim \aut(\w{Z}) =  \dim \fh^{S_x}$ for the  system of fundamental forms $S_x \subset \Sym T_x^* M$ for a point $x \in Z$ chosen in Theorem \ref{t.bound}. To prove that  $Z \subset \BP V$ is projectively equivalent to $Z^{S_x} \subset \BP V^{S_x}$, it suffices to check the three conditions in Proposition \ref{p.C*}. The condition (i)  is obvious by the linear nondegeneracy of $Z \subset \BP V$.

By $\dim \aut(\w{Z}) =  \dim \fh^{S_x}$ and Lemma \ref{l.bound},  the inclusion $ \fh_x^k/\fh_x^{k+1}  \subset  \fh^{S_x}_k$ in Theorem \ref{t.bound} must be an isomorphism for all $k \in \Z$. 
From $\fh_x^{-1}/\fh_x^0  \cong \fh^{S_x}_{-1} = T_x Z$, we see that the orbit of $x \in Z$ under the group  $ \exp \aut(\w{Z}) \subset {\rm GL}(V)$ contains an open neighborhood $O \subset Z$ of $x$. It follows that the condition (ii)  in Proposition \ref{p.C*} is satisfied by any $z \in O$. 

It is easy to see that  $${\rm Id}_{W} \in \Sym^1 W^* \otimes W = \fv_1(W) \subset \fg_1(W)$$ lies in $\fh^{S}_0 $ in Definition \ref{d.fhS}.  From  $ \fh_x^0/\fh_x^{1}  \cong  \fh^{S_x}_0,$ there must be an element $E_x \in  \fh_x^0$ representing ${\rm Id}_{T_x M}.$ Then the algebraic closure of $\exp(\C E_x)$ must contain a multiplicative subgroup $\C^* \subset {\rm Aut}(\w{Z}) \subset {\rm GL}(V)$ 
such that the induced $\C^*$-action on the tangent space $T_x Z$ is by scalar multiplication. Thus  the condition  (iii) in Proposition \ref{p.C*} is satisfied by any $z \in O.$   This proves the theorem. 
\end{proof}

\medskip
Jun-Muk Hwang (jmhwang@ibs.re.kr)
Center for Complex Geometry,
Institute for Basic Science (IBS),
Daejeon 34126, Republic of Korea

\medskip
Qifeng Li (qifengli@sdu.edu.cn)
School of Mathematics,
Shandong University,
Jinan 250100, China
 \end{document}